\documentclass{amsart}

\usepackage{amscd}

\def\e{{\epsilon}}

\author[L. Vivas]{Liz Vivas}
\address{Liz Vivas\\ Department of Mathematics\\ The Ohio State University\\
\\ 231 West 18th Ave.\\ Columbus, OH, USA} \email{vivas@math.osu.edu}

\thanks{The author was partially supported by NSF Grant DMS-1800777.}

\date{}

\title{Non-Autonomous Parabolic Bifurcation}

\begin{document}
\begin{abstract}
Let $f(z) = z+z^2+O(z^3)$ and $f_\e(z) = f(z) + \e^2$. A classical result in parabolic bifurcation in one complex variable is the following: if $N-\frac{\pi}{\e}\to 0$ we obtain $(f_\e)^{N} \to \mathcal{L}_f$, where $\mathcal{L}_f$ is the Lavaurs map of $f$. In this paper we study a \textit{non-autonomous} parabolic bifurcation. We focus on the case of $f_0(z)=\frac{z}{1-z}$. Given a sequence $\{e_i\}_{1\leq i\leq N}$, we denote $f_n(z) = f_0(z) + \e_n^2$. We give sufficient and necessary conditions on the sequence $\{e_i\}$ that imply that $f_{N}\circ\ldots f_{1} \to \textrm{Id}$ (the Lavaurs map of $f_0$). We apply our results to prove parabolic bifurcation phenomenon in two dimensions for some class of maps. 
\end{abstract}
\newtheorem{theorem}{Theorem}
\newtheorem{lemma}{Lemma}
\newtheorem{proposition}{Proposition}
\newtheorem{corollary}{Corollary}
\newtheorem{example}{Example}
\newtheorem{question}{Question}
\newtheorem*{question*}{Question}

\theoremstyle{definition}
\newtheorem{defin}{Definition}

\theoremstyle{remark}
\newtheorem{remark}{Remark}

\newcommand{\NN}{\mathbb{N}}
\newcommand{\RR}{\mathbb{R}}
\newcommand{\CC}{\mathbb{C}}
\newcommand{\PP}{\mathbb{P}}
\newcommand{\ZZ}{\mathbb{Z}}
\newcommand{\R}{\mathcal{R}}
\newcommand{\TT}{\mathbb{T}}
\newcommand{\x}{\mathbf{x}\rm}
\newcommand{\y}{\mathbf{y}\rm}
\newcommand{\zbar}{\overline{z}}
\newcommand{\wbar}{\overline{w}}
\newcommand{\ubar}{\overline{u}}
\newcommand{\vbar}{\overline{v}}
\newcommand{\Ree}{\mathrm{Re}}
\newcommand{\Imm}{\mathrm{Im}}
\newcommand{\tp}{\widetilde{p}}

\def\a{{\alpha}}
\def\b{{\beta}}
\def\d{{\delta}}
\def\e{{\epsilon}}
\def\i{{\iota}}
\def\s{{\sigma}}
\def\t{{\tau}}
\def\r{{\rho}}
\def\g{{\gamma}}
\def\l{{\lambda}}
\def\th{{\theta}}

\def\O{{\Omega}}
\def\Aut{{\mathrm{Aut}}}

\maketitle

\section{Introduction}

The theory of parabolic bifurcation has been extensively studied in one dimension starting with the pioneering work of Lavaurs and Douady, as well as Shishikura \cite{Dou, Lav, Shi}. In recent years, parabolic bifurcation has been explored in several dimensions; Bedford, Smilie and Ueda studies semi parabolic bifurcations \cite{BSU} and Bianchi \cite{Bi} has studied parabolic bifurcations for a class of maps. Also the recent works of Dujardin and Lyubich \cite{DL} and Astorg et al \cite{ABDPR} have shown applications to new phenomena in several dimensions using higher dimension parabolic bifurcations.

In this article we propose to study parabolic bifurcation in two dimensions by considering \textit{non-autonomous} sequences of one dimensional M\"oebius maps. Let us recall the result in one dimension and explain our result.

\begin{theorem} (Lavaurs) Let $f$ be defined in a neighborhood $V$ of the origin and be of the form $f(z) = z + z^2 +O(z^3)$. Consider the perturbation of $f$ as follows: given $\e>0$ let $f_\e(z) := f(z) + \e^2$. If we take a sequence of number $N_\e$ such that $N_\e - \frac{\pi}{\e} \to 0$ then we obtain the following:
$$
(f_\e)^{N_\e} \to \mathcal{L}_f,
$$
where $\mathcal{L}_f$ is the Lavaurs map of $f$.
\end{theorem}

In this paper we study the following question.

\begin{question*} Given $f$ be defined in a neighborhood $V$ of the origin and be of the form $f(z) = z + z^2 +O(z^3)$ as above. Consider different perturbations of $f$ as follows: given $\e_k>0$ let $f_k(z) := f(z) + \e_k^2$. 
Under which conditions on $\e_1,\e_2,\ldots,\e_N$ do we obtain:
\begin{align}\label{nonauto}
f_N \circ f_{N-1} \circ \ldots f_2\circ f_1 \to \mathcal{L}_f,
\end{align}
where $\mathcal{L}_f$ is the Lavaurs map of $f$?
\end{question*}
When \eqref{nonauto} holds we say that `non-autonomous bifurcation' holds for $f$.

\strut

In this paper we focus on the case of $f(z) =\frac{z}{1-z}$ and prove a sufficient and necessary condition on $\e_n$ for the following:

\begin{theorem}\label{maintheoremintro} Fix $N$ large. Consider $\{\e_k\}, 1 \leq k \leq N$ a sequence such that:
\begin{align}
\e_k = \frac{\pi}{N} + \frac{\alpha(k)}{N^2}+O\left(\frac{1}{N^3}\right)
\end{align}
where $\alpha(k)$ are bounded and $\alpha(k)+\alpha(N-k) =O(1/N)$. 
Let $f_k(z) := \frac{z}{1-z}+\e_k^2$. Then we have:
\begin{align}\label{nonautotheo}
f_N \circ f_{N-1} \circ \ldots f_2\circ f_1 \to \textrm{Id},
\end{align}
for all $z\in K$, where $K$ compact subset of $\CC$.
\end{theorem}

We call this result `non-autonomous bifurcation' for a M\"obius transformation, since in this case the Lavaurs map of $f$ is simply the identity. We prove also that the condition is necessary (see Section \ref{necessaryconditions})

%Let us explain now how we can connect this results to a larger class of maps following the ideas of McMullen \cite{McM}.

\begin{remark}
McMullen studied classic bifurcations for general maps by focusing on M\"obius transformations. In \cite{McM} he proves that, even when a general parabolic map $f$ is not analytically conjugated to the map $T(z)= z+1$ (consider the fixed point here to be infinity), we can still find a quasiconformal conjugacy $\phi$. Even more, if we have a sequence of maps $f_n \to f$, under certain conditions, there exists a sequence of quasi-conformal maps $\phi_n$ such that $f_n$ is conjugated by $\phi_n$ to $T_n(z) = \l_n(z+1)$ and $\phi_n \to \phi$ as well as $T_n \to T$ (see Theorem 8.2 in \cite{McM}). Using these facts, and proving parabolic bifurcation for $T_n$ (that is that $T_n^M$ has a limit under certain conditions on the relationship between $M$ and $n$), McMullen proves that $(f_n)^M$ will also have a limit. However, when we try to apply the same ideas in our setting, we do not obtain a limit for a general limit for $f_N\circ\ldots f_1$ since the quasiconformal conjugacy can vary for each $f_N$. It would be interesting to see if we still could obtain non-autonomous bifurcation for a general parabolic $f$ using the theory of quasiconformal mappings as done by McMullen. 
\end{remark}

\begin{remark}
In order to prove our main result we use the theory of orthogonal functions. This is to our knowledge the first time that we have a connection to this field. More importantly, we only use a particular version of a general estimate for orthogonal polynomials. We believe that the general version of this theorem must have its corresponding bifurcation version. See Section 2.1 for more details.
\end{remark}

The structure of the paper is as follows. In the next section we set the notation and well as our theorems and the proofs. In section 3, we give examples of sequences that satisfying the conditions. In section 4 we apply our results to prove parabolic bifurcation for specific families of maps in two dimensions. In the last section, we formulate some questions and remarks.\\

\textit{Acknowledgements} The author would like to thank Han Peters for comments in a preliminary version of the paper. Part of this work was done while the author visited Oberwolfach for the workshop Geometric Methods of Complex Analysis in August 2018. Thanks to this institution and the organizers for great working conditions.
. 

\section{Non-Parabolic Bifurcation}

Given a sequence of positive real numbers $\e_1,\e_2,\ldots$. Consider the following functions
$$
f_k (z) = f_{\e_k}(z) = \frac{z}{1-z} + \e_k^2,
$$
for $k\geq 1$.
Set
$$
F_n = f_n\circ f_{n-1}\circ \ldots f_2\circ f_1.
$$
 
We prove the following technical version of our theorem. We see at the end of this section how Theorem \ref{maintheorem} implies Theorem \ref{maintheoremintro}.
\begin{theorem}\label{maintheorem} Fix $N$ large. Consider $\{\e_k\}, 1 \leq k \leq N$ a sequence such that:
\begin{align}
\left|\sum_{k=1}^{N} \left(\frac{\pi^2}{N^2}-\e_k^2\right)\frac{\sin{(k\pi/N})^2}{\sin(\pi/N)^2} \right| &< \frac{A}{N}
\end{align}
and
\begin{align}
\left|\frac{\pi^2}{N^2}- \e_k^2 \right| &< \frac{A}{N^3}\nonumber
\end{align}
for all $1 \leq k \leq N$ and a fixed constant $A$ independent of $N$.
%$$
%N^2\sqrt{\frac{\sum_{n=1}^N \e_n^2}{N}} - N\pi \to 0,
%$$
Then:
$$
|F_N(z) - z| < C/N
$$
for all $z\in K$, where $K$ compact subset of $\CC$ and some $C$ independent of $N$.
\end{theorem}

(Note that when $\e_k=\e=\pi/N$ then the conditions on the theorem are satisfied trivially. The conclusion that $\lim_{N \to\infty} (f_\e)^N(z) = z$ is a particular case of the classical bifurcation theorem in one dimension.)

Fix $N \geq 1$. Since each $f_n$ is a M\"oebius transformation then we can compute the specific formula for $F_k$ by computing the product of the matrices related to each.

\begin{equation}
F_k(z) = \frac{A_k z+C_k}{B_k z+D_k},
\end{equation}
then:
$$
\left({\begin{array}{cc} A_k & C_k\\ B_k & D_k \end{array}}\right) = 
\left({\begin{array}{cc} 1-\e_k^2 & \e_k^2\\ -1 & 1 \end{array}}\right) \left({\begin{array}{cc} A_{k-1} & C_{k-1}\\ B_{k-1} & D_{k-1} \end{array}}\right) .
$$

\begin{lemma} Set $t_k=2-\e_k^2$ the trace of each matrix $A_k$.
Consider the sequence $p_{0}=0, p_1=1$ and $q_{0}=1,q_1=1$ and for $k\geq 1$:
\begin{align}\label{gencheb}
p_{k+1}=t_{k}p_k-p_{k-1}\\
q_{k+1}=t_{k}q_k-q_{k-1}. \nonumber
\end{align}
Then for any $n\geq 1$ we have:
\begin{align*}
A_n &=p_{n+1}-p_{n},  C_n =q_{n}-q_{n+1}\\
B_n &=-p_{n},  \qquad D_n =q_{n}.
\end{align*}
\end{lemma}

\begin{proof}[Proof of Lemma]
it follows directly using induction. 
\end{proof}

Although the following statement is, as mentioned above a particular case of the general parabolic bifurcation in one variable, we redo the proof here as a preparation step for the proof of Theorem \ref{maintheorem}.

\begin{lemma} Fix $N$. Suppose that all $\e_i = \e$ and the condition:
\begin{align*}
N-\pi/\e \to 0
\end{align*}
Then $F_N(z) \to \textrm{Id}$.
\end{lemma}
\begin{proof}
The equation \eqref{gencheb} is a generalization of the classical Chebyshev polynomials. Note that the classical Chebyshev polynomial corresponds to the case of the same $\e_i$, that is, the classical parabolic bifurcation on one variable. Indeed, if we have all $\e_i=\e$ equals and $t_i=x=2-\e^2$, then it's well known that:
$$
p_k = \frac{\sin (k\theta)}{\sin(\theta)} \qquad\textrm{ and }q_k= \frac{\sin(k \theta)-\sin((k-1)\theta)}{\sin(\theta)},
$$
where $x=2\cos(\theta)$. When $x=2-\e^2$, then $\theta = \e + O(\e^3)$. Suppose $N-\pi/\e \to 0$. On that case $N-\pi/\e = o_N(1)$ and we can write $\theta = \frac{\pi}{N}+\frac{\pi o_N(1)}{N^2}+O\left(\frac{1}{N^3}\right)$. Then
\begin{align*}
p_N = \frac{\sin(N\theta)}{\sin(\theta)}=\frac{\sin\left(\pi + \frac{\pi o_N(1)}{N}+O\left(\frac{1}{N^2}\right)\right)}{\sin\left(\frac{\pi}{N}+\frac{\pi o_N(1)}{N^2}+O\left(\frac{1}{N^3}\right)\right)} = o_N(1)\\
p_{N+1} = \frac{\sin((N+1)\theta)}{\sin(\theta)}=\frac{\sin\left(\pi + \frac{\pi}{N}+O\left(\frac{\pi 0_N(1)}{N}\right)\right)}{\sin\left(\frac{\pi}{N}+\frac{\pi o_N(1)}{N^2}+O\left(\frac{1}{N^3}\right)\right)} = -1+ o_N(1)\\
\end{align*}
and similarly $p_{N-1} = 1+ o_N(1)$, which translated to the element of our matrix:
$$
A_N = D_N = -1+o_N(1), B_N = C_N = o_N(1). 
$$
Therefore when $N \to \infty$,  $A_N = D_N \to -1, B_N = C_N \to 0$ so $F_N(z) \to \textrm{Id}$.
\end{proof}

As it is clear from the proof of the lemma above, if we have estimates on $p_N$ and $q_N$ then we immediately have the estimates for $A_N,B_N,C_N,D_N$. 

\subsection{Orthogonal polynomials}
We review here some facts about orthogonal polynomials. We use the following lemma from \cite{Nev}. 

\begin{lemma}
Consider the sequence $p_{0}=0, p_1=1$ and for $k\geq 1$:
\begin{align}\label{generalnevai}
p_{k+1}=(x+a_k)p_k-p_{k-1}
\end{align}
Let $x=2\cos(\theta)$, then we have the following equality:
\begin{align}\label{mainformula}
\sin(\theta)p_n(x) = |\phi_{n}| \sin(n\theta - \arg(\phi_n))
\end{align}
where
$\phi_n = 1 +  \delta_n = 1+ \sum_{j=1}^{n-1}a_jp_je^{ij\theta}$ for $n\geq 2$ and $\phi_1=1$.
\end{lemma}

For simplification we will use the following terminology for the classical Chebyshev polynomials $U_{0}=0,U_1=1$ and for $k\geq 1, U_{k+1}=xU_k-U_{k-1}$. In that case $U_k = \sin(k\theta)/\sin(\theta)$ for $x=2\cos(\theta)$.

\begin{lemma}
Consider the following two sequences:
\begin{align*}
p_{0}=0, &p_1=1, p_{k+1}=(x+a_k)p_k-p_{k-1}, k\geq 1\\
U_{0}=0, &U_1=1, U_{k+1}=xU_k-U_{k-1}, k\geq 1
\end{align*}
Let $x=2\cos(\theta)$. Suppose there exists $\e>0$ and $m \in \NN$ such that the sequence $\{a_i\}$ satisfies:
\begin{align}
\sum_{j=1}^{m-1}|a_jp_j| \leq \e\sin(\theta)
\end{align} 
then 
\begin{align*}
|p_n-U_n| \leq \e,
\end{align*}
for all $1\leq n \leq m$.
\end{lemma}

\begin{proof} We use equation \eqref{mainformula}:
\[
\sin(\theta)p_n(x) =  \sin(n\theta)(1+\Ree(\delta_n)) - \cos(n\theta)\Imm(\delta_n)
\]
\[
\sin(\theta)p_n(x) =  \sin(n\theta)+\sin(n\theta)\Ree(\delta_n) - \cos(n\theta)\Imm(\delta_n)
\]
\[
\sin(\theta)(p_n - U_n) = -\Imm(\delta_n e^{-in\theta})
\]
Recall that $\delta_n = \sum_{j=1}^{n-1}a_jp_je^{ij\theta}$ then $|\delta_n| \leq \sum_{j=1}^{n-1}|a_jp_je^{ij\theta}|= \sum_{j=1}^{n-1}|a_jp_j|<\e \sin(\theta)$
and we obtain immediate the desired result.

\end{proof}

\subsection{Proof of Theorem 3}
We are ready now to prove Theorem 3. Fix $N>0$ large, we use the lemmas referred above with the following choices: $x = 2\cos(\theta)$ where $\theta = \frac{\pi}{N}$, then we have explicit values and estimates for $U_i$ for all $i$ in terms of $N$. In particular $|U_i| \leq \frac{1}{\sin(\theta)} \leq \frac{2N}{\pi}\leq N$. 
Our goal is to prove that under certain conditions on $a_k$ then $p_n$ and $U_n$ are very close to each other. 

\begin{lemma}
Fix $N>0$. Let $x = 2\cos(\theta)$ where $\theta = \frac{\pi}{N}$. Given a sequence $\{a_i\}$ for $1\leq i\leq N$. Suppose there exists a fixed $C>0$ constant such that
\begin{align}
|a_i| \leq \frac{C}{N^3} \leq \frac{1}{2N^2}
\end{align} 
then we have $|p_i-U_i| < 2C$ for $1 \leq i \leq N+1$.
\end{lemma}

\begin{proof}
From the proof of the last lemma we have
\[
\frac{1}{N}|p_n - U_n| \leq |\sin(\theta)||p_n - U_n| \leq |\delta_n| \leq \sum_{j=1}^{n-1}|a_jp_j|.
\]
We use induction. The is obvious for $i=1$. Assume the bound holds for $i \in [1,n-1]$, then for $i=n\leq N+1$ we have
\begin{align*}
\sum_{j=1}^{n-1}|a_jp_j| &\leq \sum_{j=1}^{n-1}|a_jU_j| + 2C\sum_{j=1}^{n-1}|a_j|\\
&\leq N\sum_{j=1}^{n-1}|a_j| + 2C\sum_{j=1}^{n-1}|a_j|\\
&\leq (N+2C)(n-1)\frac{C}{N^3}\\
&\leq (N+2C)\frac{C}{N^2}\\
\end{align*}
Then
\[
\frac{1}{N}|p_n - U_n| \leq (N+2C)\frac{C}{N^2}
\]
\[
|p_n - U_n| \leq (N+2C)\frac{C}{N} = C + \frac{2C^2}{N} \leq C +C = 2C.
\]
which concludes the proof.
\end{proof}

\begin{proposition}\label{pestimates} Fix $N$. Let $x=2\cos\frac{\pi}{N}$. Suppose that
\begin{align}
\left|\sum_{k=1}^{N}  a_k U_k^2 \right| \leq \frac{C}{N} \textrm{ and }|a_k|\leq \frac{C}{N^3} 
\end{align}
for all $1 \leq k \leq N$ and a fixed constant $C$ independent of $N$.
Then:
$$
|p_N|\leq C'/N,  |p_{N+1}+1|\leq C'/N,
$$
for some $C'$ independent of $N$.
\end{proposition}

\begin{proof}
We use  \eqref{mainformula} for $\theta=\pi/N$:
\begin{align*}
\sin(\theta)p_N &= |\phi_{N}| \sin(N\theta - \arg(\phi_N))\\
&= |\phi_{N}| \sin(\pi - \arg(\phi_N)) \\
&=|\phi_{N}| \sin(\arg(\phi_N)) = \Imm(\phi_N) = \Imm(\delta_N)\\
&= a_1p_1\sin(\theta)+a_2p_2\sin(2\theta)+\ldots+a_{N-1}p_{N-1}\sin((N-1)\theta)
\end{align*}
Then
\begin{align*}
|p_N|
&= \frac{1}{\sin(\theta)}|a_1p_1\sin(\theta)+a_2p_2\sin(2\theta)+\ldots+a_{N-1}p_{N-1}\sin((N-1)\theta)|\\
&=|a_1p_1U_1+a_2p_2U_2+\ldots+a_{N-1}p_{N-1}U_{N-1}|\\
&\leq|a_1U_1^2+a_2U_2^2+\ldots+a_{N-1}U_{N-1}^2| + 2C\sum_{i=1}^{N-1}|a_iU_i|\\
&\leq\frac{C}{N} + 2C\frac{C}{N^3}N^2 = \frac{C'}{N}.
\end{align*}

Similarly for $p_{N+1}$ we have:
$$
\sin(\theta)(p_{N+1} - U_{N+1}) = -\Imm(\delta_{N+1} e^{-i\pi - i\theta}) = \Imm(\delta_{N+1} e^{- i\theta}) 
$$
where 
$$
\delta_{N+1} = \sum_{k=1}^N a_kp_ke^{ik\theta}
$$
then
$$
e^{-i\theta}\delta_{N+1} = \sum_{k=1}^N a_kp_ke^{i(k-1)\theta}
$$
so
$$
\Imm(\delta_{N+1} e^{- i\theta}) = \sum_{k=1}^N a_kp_k \sin((k-1)\theta)
$$
and we obtain:
$$
p_{N+1} - U_{N+1} =  \sum_{k=2}^N a_kp_k U_{k-1}
$$
Using the fact that $\left|\sum_{k=1}^{N}  a_k U_k^2 \right| < \frac{C}{N}$ implies that $\left|\sum_{k=1}^{N}  a_k U_kU_{k-1} \right| < \frac{C''}{N}$ and the same idea that for $p_N$ we obtain
$$
|p_{N+1} - U_{N+1}| = |p_{N+1} +1| < \frac{C'}{N}.
$$
\end{proof}

We are almost done proving Theorem 1, however we still need analogue bounds for $q_n$. 
\begin{lemma}
Consider the sequences $p_{0}=0, p_1=1$ and $q_{0}=1,q_1=1$ and for $k\geq 1$:
\begin{align*}
p_{k+1}=t_{k}p_k-p_{k-1}\\
q_{k+1}=t_{k}q_k-q_{k-1}. 
\end{align*}
Then
$$
q_k = p_k - \tp_{k-1}
$$
where the sequence $\tp_k$ is given by the conditions $\tp_{0}=0, \tp_1=1$ and for $k\geq 1$ we have 
\begin{align*}
\tp_{k+1}=t_{k+1}\tp_k-\tp_{k-1}.
\end{align*}
\end{lemma}
\begin{proof}
The proof follows immediately by writing down $q_k-p_k$ and checking the corresponding initial conditions.
\end{proof}

Using the same idea and estimates for $p_N$ we have the following
\begin{proposition} Fix $N$. Let $x=2\cos\frac{\pi}{N}$. Suppose that
\begin{align}
\left|\sum_{k=1}^{N-1}  a_{k+1} U_k^2 \right| \leq \frac{C}{N} \textrm{ and }|a_k|\leq \frac{C}{N^3} 
\end{align}
for all $1 \leq k \leq N$ and a fixed constant $C$ independent of $N$.
Then:
$$
|\tp_N|\leq C'/N,  |\tp_{N-1}-1|\leq C'/N,
$$
for some $C'$ independent of $N$.
\end{proposition}

\begin{proof}
The proof is exactly the same than the proof of Proposition \ref{pestimates}; the only difference pertains the shifted terms which envolves the $a_i's$.
\end{proof}

We are ready now to combine all the lemmas above and finish the proof of Theorem \ref{maintheorem}.\\

Given a sequence $\{\e_k\}$ such that:
\begin{align*}
\left|\sum_{k=1}^{N} \left(\frac{\pi^2}{N^2}-\e_k^2\right)\frac{\sin{(k\pi/N})^2}{\sin(\pi/N)^2} \right| < \frac{A}{N}
\end{align*}
and
$$
\left|\frac{\pi^2}{N^2}- \e_k^2 \right| < \frac{A}{N^3}
$$
for all $1 \leq k \leq N$ and a fixed constant $A$ independent of $N$.
Notice that we can write this in terms of $x=2\cos(\pi/N)$ and $a_k = t_k - x$ where $t_k = 2-\e_k^2$, so we obtain:
$$
\left|\sum_{k=1}^{N} a_kU_k^2\right| < \frac{A}{N}
$$
and
$$
|a_k| < \frac{C}{N^3}
$$
for all $1\leq k\leq N$.
Using lemma 1, Proposition 2 and Proposition 3, we see that $A_N = D_N = -1 + O(1/N)$ and $B_N = C_N = O(1/N)$. Which translating back into $F_N$ implies that $F_N(z)\to \textrm{Id}$ when $N\to\infty$.

\subsection{Proof of Theorem \ref{maintheoremintro}}

All is left to prove is that the conditions of $\e_k$ on Theorem \ref{maintheoremintro} are satisfied for Theorem \ref{maintheorem}.
Given $\e_k$ such that
\begin{align*}
\e_k = \frac{\pi}{N} + \frac{\alpha(k)}{N^2}+O\left(\frac{1}{N^3}\right)
\end{align*}
where $\alpha(k)$ are bounded then we immediately have:
$$
\left|\frac{\pi^2}{N^2}- \e_k^2 \right| < \frac{A}{N^3}.
$$
Also 
$$
\frac{\pi^2}{N^2}- \e_k^2 = \frac{-2\pi\alpha(k)}{N^3}+O\left(\frac{1}{N^4}\right)
$$
Therefore:
$$
S=\left|\sum_{k=1}^{N} \left(\frac{\pi^2}{N^2}-\e_k^2\right)\frac{\sin{(k\pi/N})^2}{\sin(\pi/N)^2} \right| =\left|\sum_{k=1}^{[N/2]} \left(\frac{-2\pi(\alpha(k)+\alpha(N-k))}{N^3}+O\left(\frac{1}{N^4}\right)\right)\frac{\sin{(k\pi/N})^2}{\sin(\pi/N)^2} \right| 
$$
Since we have the condition $\alpha(k)+\alpha(N-k) = O(1/N)$ then we have
$$
S= \left|\sum_{k=1}^{[N/2]} O\left(\frac{1}{N^4}\right)\frac{\sin{(k\pi/N})^2}{\sin(\pi/N)^2} \right| <\frac{C}{N^4}.[N/2].N^2 = \frac{C'}{N},
$$
where we are using the trivial bounds on each $\frac{\sin{(k\pi/N})^2}{\sin(\pi/N)^2}<N^2$ and adding the $N/2$ factors.
We have that both conditions of Theorem \ref{maintheorem} are satisfied and the conclusion follows.

\subsection{Conditions are necessary}\label{necessaryconditions}

In this section we prove the following result:
\begin{theorem}\label{counterexample} There exists $\{\e_k\}, 1 \leq k \leq N$ a sequence such that:
\begin{align}
\e_k = \frac{\pi}{N} + \frac{\alpha(k)}{N^2}+O\left(\frac{1}{N^3}\right)
\end{align}
where $\alpha(k)$ are bounded and for $f_k(z) := \frac{z}{1-z}+\e_k^2$ we have
$$
f_N \circ f_{N-1} \circ \ldots f_2\circ f_1 \not\to \textrm{Id},
$$
\end{theorem}

\begin{proof}
Fix $N$ as above. Let $\e_k^2 = 2-2\cos(\frac{\pi}{N+1})$ for all $k$. We have that each $\e_k = \frac{\pi}{N}- \frac{\pi}{N^2}+O(1/N^3)$. So the conditions above are satisfied. However we do know that
$f_{N+1} \circ f_{N} \circ \ldots f_2\circ f_1 = (f_\e)^{N+1}=\textrm{Id}$ and therefore $f_\e^N= (f_e)^{-1}\to \frac {z}{1+z}$ when $N\to \infty$.
\end{proof}
%%%%%%%%%%%%%%%%%%%%%%%%%%%%%%%%%%%%%%%%%%%%%%%%%%%
%%%%%%%%%%%%PARTICULAR CASES%%%%%%%%%%%%%%%%%%%%%%
%%%%%%%%%%%%%%%%%%%%%%%%%%%%%%%%%%%%%%%%%%%%%%%%%%%

\section{Special examples}

\subsection{Perturbations of the autonomous case}
\begin{theorem}\label{closebytheorem}
Fix $N>0$ and a sequence of positive real numbers $\{\e_k, 1\leq k\leq N\}$ satisfying the following condition:
\begin{align*}
\e_k=\frac{\pi}{N}+ A\left(-\frac{1}{N^2} + \frac{2k}{N^3}\right)+ O\left(\frac{1}{N^3},\right)
\end{align*}
for $1\leq k\leq N$, and a constant $A$ independent of $N$.
Then we have that the following holds:
$$
F_N = f_N\circ f_{N-1}\circ \ldots f_2\circ f_1 = z + \frac{B(z)}{N}
$$
where $f_k (z) = f_{\e_k}(z) = \frac{z}{1-z} + \e_k^2$, for $k\geq 1$.
\end{theorem}

\begin{proof}
Note that 
\[
a_k = 2-\e_k^2-2\cos\left(\frac{\pi}{N}\right) = \frac{\pi^2}{N^2} - \e_k^2 +O(1/N^4)
\]
and given the condition on $\e_k$ we have therefore that
$$
a_k= \frac{2A\pi}{N}\left(\frac{1}{N^2} - \frac{2k}{N^3}\right) +O\left(\frac{1}{N^4}\right); 
|a_k| \leq \frac{C'}{N^3}.
$$
As well as: 
\begin{align*}
\left|\sum_{k=1}^{N} \left(2-\e_k^2-2\cos\frac{\pi}{N}\right)\frac{\sin{(k\pi/N})^2}{\sin(\pi/N)^2} \right|&= \left|\sum_{k=1}^{N} a_kU_k^2 \right| \\
&=\left|\sum_{k=1}^{\lfloor{N/2}\rfloor}  (a_k+a_{N-k}) U_k^2\right|.
\end{align*}
where we use $U_k= U_{N-k}$.
Since
$$
a_k= \frac{2A\pi}{N^4}(N - 2k) +O\left(\frac{1}{N^4}\right); 
$$
Then $a_k + a_{N-k}= O\left(\frac{1}{N^4}\right)$ therefore 
\begin{align*}
\left|\sum_{k=1}^{\lfloor{N/2}\rfloor}  (a_k+a_{N-k}) U_k^2\right| \leq \frac{C'}{N}
\end{align*}
and both conditions of Theorem \ref{maintheorem} so the theorem is proved.
\end{proof}

\begin{example}

Given $m\in\NN$, consider the following sequence:
\begin{align}\label{epsilonwandering1}
\e_k=\frac{\pi}{2\sqrt{m^2+k}}
\end{align}
for $1\leq k \leq 2m+1=N$.
Then $N-1=2m$ and:
$$
\e_k=\frac{\pi}{\sqrt{(N-1)^2+4k}}= \frac{\pi}{N}\left(1-\frac{(2-4k/N)}{N}+\frac{1}{N^2}\right)^{-1} =  \frac{\pi}{N}-\frac{2\pi(2k/N-1)}{N^2}+O(1/N^3)
$$
So we have $|a_k+a_{N-k}|<\frac{C}{N^4}$ for all $1 \leq k \leq N$.
and Theorem \ref{closebytheorem} with $A=-2\pi$ applies.
\end{example}

\begin{align*}
\e_k=\frac{\pi}{N}+ A\left(-\frac{1}{N^2} + \frac{2k}{N^3}\right)+ O\left(\frac{1}{N^4},\right)
\end{align*}

\begin{example}

Given $m\in\NN$, consider the following sequence:
\begin{align}\label{epsilonwandering2}
\e_k=\frac{\pi}{2\sqrt{4m^2+2k}}
\end{align}
for $1\leq k \leq 4m+2=N$.
Then, a similar computation as above shows that:
$$
\e_k = \frac{\pi}{\sqrt{16m^2+8k}}= \frac{\pi}{\sqrt{(N-2)^2+8k}}=  \frac{\pi}{N}-\frac{2\pi(2k/N-1)}{N^2}+O(1/N^3)
$$
and we can apply Theorem \ref{closebytheorem} again.
\end{example}

%Then we have the following:
%\begin{align}
%\sum_{k=1}^{2n}\e_k^2 = \frac{\pi^2}{2n}\left(1-\frac{1}{n}+\frac{5}{6n^2}+O\left(\frac{1}{n^3}\right)\right)\\
%\sum_{k=1}^{2n+1}\e_k^2 = \frac{\pi^2}{2n}\left(1-\frac{1}{2n}-\frac{1}{6n^2}+O\left(\frac{1}{n^3}\right)\right).
%\end{align}
%
%Using the definition before:
%$$
%\delta_{m}:=\sqrt{\frac{1}{m}\sum_{i=1}^{m}\e_i^2}
%$$
%then we obtain:
%\begin{align}
%\delta_{2k}= \frac{\pi}{2k}\left(1-\frac{1}{2k}+\frac{7}{24k^2}+O\left(\frac{1}{k^3}\right)\right)\\
%\delta_{2k+1}= \frac{\pi}{2k}\left(1-\frac{1}{2k}-\frac{1}{12k^2}+O\left(\frac{1}{k^3}\right)\right).
%\end{align}
%Then:
%\begin{align}
%(2n+1)\delta_{2n}= \pi\left(1+\frac{1}{24n^2}+O\left(\frac{1}{n^3}\right)\right)\\
%(2n+1)\delta_{2n+1}= \pi\left(1-\frac{1}{3n^2}+O\left(\frac{1}{n^3}\right)\right).
%\end{align}
%So:
%\begin{align}
%\sin((2n+1)\delta_{2n+1})= O\left(\frac{1}{n^2}\right).
%\end{align}

\subsection{Very close perturbations}

\begin{theorem}\label{secondtheorem}
Fix $N>0$ and a sequence of positive real numbers $\{\e_k, 1\leq k\leq N\}$ satisfying the following condition:
\begin{align*}
\left|\e_k-\frac{\pi}{N}\right| \leq \frac{C}{N^3},
\end{align*}
for a constant $C$ independent of $N$.
Then we have that the following holds:
$$
F_N = f_N\circ f_{N-1}\circ \ldots f_2\circ f_1 = z + \frac{A(z)}{N}
$$
where $f_k (z) = f_{\e_k}(z) = \frac{z}{1-z} + \e_k^2$, for $k\geq 1$.
\end{theorem}

\begin{proof}
Note that 
\[
a_k = 2-\e_k^2-2\cos\left(\frac{\pi}{N}\right) = \frac{\pi^2}{N^2} - \e_k^2
\]
and given the condition on $\e_k$ we have therefore that
$$
|a_k| \leq \frac{C'}{N^4}.
$$
So both condition on Theorem \ref{maintheorem} are satisfied. Indeed, the second condition is clear, and the first one follows since each sine term is bounded by $1$ above then:
\begin{align*}
\left|\sum_{k=1}^{N} \left(2-\e_k^2-2\cos\frac{\pi}{N}\right)\frac{\sin{(k\pi/N})^2}{\sin(\pi/N)^2} \right|&= \left|\sum_{k=1}^{N} a_k\frac{\sin{(k\pi/N})^2}{\sin(\pi/N)^2} \right| \\
&\leq  \left|\sum_{k=1}^{N} a_k\frac{1}{\sin(\pi/N)^2} \right|\\
&\leq  \left|N \frac{C'}{N^4}\frac{4N^2}{\pi^2} \right|\\
&< \frac{A}{N}
\end{align*}
\end{proof}

\begin{example}
Given $N\in\NN$, consider the following sequence:
$$
\e_k=\frac{\pi}{(N^3+k)^{1/3}}
$$
for $1\leq k \leq N$.
Then
$$
\e_k =\frac{\pi}{N}\left(1+\frac{k}{N^3}\right)^{-1/3}
\sim \frac{\pi}{N} - \frac{\pi k}{3N^4} + O\left(\frac{k^2}{N^7}\right).
$$
Then Theorem \ref{secondtheorem} applies and we have the result for this specific choice of $\e_k$.
\end{example}

\section{Bifurcations for two dimensional maps}

Much of this works has been inspired by the recent paper by Astorg, Buff, Dujardin, Peter and Raissy \cite{ABDPR} on bifurcations for a specific map on two dimensions.
Let us recall one part of their result.
Given the map:
$$
F(z,w)= (z+z^2 +az^3 + \frac{\pi^2}{4}w,w-w^2 +w^3)
$$
then they prove that the following holds: the sequence of maps $F^{\circ 2n+1}(z,g^{\circ n^2}(w))$ converges locally uniformly to the map $(\mathcal{L}_f(z),0)$. Here $\mathcal{L}_f$ is the Lavaurs map corresponding to the map $f$ where $F(z,0)=(f(z),0)$.

We see now that applying the same idea and prove that:
\begin{corollary} For the map
$$
H(z,w)= \left(\frac{z}{1-z} + \frac{\pi^2}{4}w, w-w^2 +w^3\right) = (h_w(z),w-w^2+w^3)
$$
then the sequence of maps $H^{\circ 2n+1}(z,g^{\circ n^2}(w))$ converges locally uniformly to the map $(z,0)$. 
As a consequence, the sequence $(H^{\circ n^2})_{n\geq 0}$ converges locally uniformly to $(\pi_z,0)$ on $\CC\times \mathcal{B}_g$, where $\pi_z$ is the projection to the first coordinate and $\mathcal{B}_g$ is the parabolic basin of $g$.
\end{corollary}
\begin{proof}
Indeed, we will have that $\frac{1}{w_k}=\frac{1}{w_0}+k+O(1/k)$ and obtain immediately that $w_k \sim 1/k$. So, if we denote $w_{n^2}=g^{\circ n^2}(w)$ and we write $h_j:=h_{w_j}$then
\begin{align*}
H^{\circ 2n+1}(z,g^{\circ n^2}(w))=  H^{\circ 2n+1}(z,w_{n^2})=  (h_{n^2+2n}\circ\ldots\circ h_{n^2+1}\circ h_{n^2}(z),w_{n^2+2n+1})
\end{align*}
where each $h_k(z)$ is as follows:
$$
h_k(z) = \frac{z}{1-z}+\frac{\pi^2}{4}w_k=\frac{z}{1-z}+\frac{\pi^2}{4k}.
$$ 
If we rename $f_1=h_{n^2}, f_2=h_{n^2+1},\ldots, f_{2n+1}=h_{n^2+2n}$ then:
$$
h_{n^2+2n}\circ\ldots\circ h_{n^2+1}\circ h_{n^2}(z) = f_{2n+1}\circ\ldots\circ f_{2}\circ f_{1}(z) 
$$
and
$$
f_k(z) = h_{n^2+k-1}(z)=\frac{z}{1-z}+\frac{\pi^2}{4(n^2+k-1)},
$$
and we see that this reduces to our example 1. Indeed, each $\e_k$ is precisely chosen to be so that $\e_k^2= \frac{\pi^2}{4(n^2+k-1)}$.
\end{proof}

Now, we use example $2$ to prove that a similar construction applies when we change the coefficient in front of the $w$ term on the first coordinate:

\begin{corollary} For the map
$$
L(z,w) = \left(\frac{z}{1-z} + \frac{\pi^2}{8}w, w-w^2 +w^3\right) = (l_w(z),w-w^2+w^3)
$$
then the sequence of maps $L^{\circ 4n+2}(z,g^{\circ 2n^2}(w))$ converges locally uniformly to the map $(z,0)$. 
As a consequence, the sequence $(L^{\circ 2n^2})_{n\geq 0}$ converges locally uniformly to $(\pi_z,0)$ on $\CC\times \mathcal{B}_g$, where $\pi_z$ is the projection to the first coordinate and $\mathcal{B}_g$ is the parabolic basin of $g$.
\end{corollary}
\begin{proof}
The proof follows exactly as before. The $\e_k$ in this case will be as chosen on \eqref{epsilonwandering2}.
\end{proof}

\section{Final remarks and questions}

\begin{remark}
It would be interesting to see if we could deduce that the result still holds even if the initial map is not the M\"oebius transformation $z/(1-z)$ but instead any other parabolic map. On the case of the classical bifurcation in one dimension, McMullen proved that this is indeed the case \cite{McM}, however it is not clear to us that the same technique can be used when we have different $\e_i$'s.
\end{remark}

%\begin{remark} For certain choices of $\e_i$ we can construct maps whose orbits have the same asymptotes as $\e_i$. Indeed, Case I above is the non-autonomous analogue of the two dimensional map $(z,w) \to (f(z)+cw^2,g(w))$ where $g$ is such that $g^k(z)\sim \e_k$. This is the map used on \cite{ABDPR}. Similary, Case 2 above is basically the non-autonomous analogue of the map $(z,w) \to (f(z)+cw^3,g(w))$. We conjecture that parabolic bifurcation therefore should also hold for a general tangent to the identity map $f$.
%\end{remark}

\begin{remark} Notice that our starting point for estimates was the estimates on Lemma 3.  That lemma holds for more general Chebyshev generalized polynomials. Let us expand a little more here.
Suppose we are give the sequence $p_0=0, p_1=1$ and $p_k = (x+a_k)p_k - \frac{b_k}{b_{k-1}}p_{k-1}$. Then similar estimates (as in Lemma 3) are obtained for this sequence. Notice that the case $b_k=1$ is the case studied here. However the case of $b_k$ not necessarily equal than $1$ has also similar estimates that will allow us to conclude that $p_N$ and $U_N$ are $O(1/N)$ distant from each other. 
Those more general sequences correspond to more general matrix products, which in principle would allow to have parabolic bifurcations not only for additive perturbations but also for multiplicative perturbations. We hope to study this case in the near future.
\end{remark}

\end{document}